\documentclass[a4paper,twoside,12pt]{article}
\usepackage[a4paper,margin=3cm,centering,nohead,ignorefoot,footskip=5ex]{geometry}
\usepackage{amsmath,amsthm}

\theoremstyle{plain}
\newtheorem*{theorem}{Theorem}
\theoremstyle{definition}
\newtheorem*{definition}{Definition}

\newcommand*{\kc}{\mathrm K}
\newcommand*{\ku}[1]{#1^\kc}
\newcommand*{\kd}[1]{#1_\kc}
\newcommand*{\kvc}[1]{\mathrm K_#1}
\newcommand*{\kvu}[2]{#2^{\kvc#1}}
\newcommand*{\kvd}[2]{#2_{\kvc#1}}
\newcommand*{\tvc}[1]{\mathrm T_#1}
\newcommand*{\tvu}[2]{#2^{\tvc#1}}
\newcommand{\cl}{\mathsf{CL}}
\newcommand{\clp}{\mathsf{CL}'}
\newcommand{\il}{\mathsf{IL}}
\newcommand{\ml}{\mathsf{ML}}
\newcommand*{\defequiv}{\mathrel{\mathop:}\equiv}
\newcommand*{\Forall}[1]{\forall #1 \,}
\newcommand*{\Exists}[1]{\exists #1 \,}

\begin{document}

\title{Variants into minimal logic of the\\Kuroda negative translation}
\author{Jaime Gaspar\thanks{INRIA Paris-Rocquencourt, $\pi r^2$, Univ Paris Diderot, Sorbonne Paris Cit\'e, F-78153 Le Chesnay, France. \texttt{mail@jaimegaspar.com}, \texttt{www.jaimegaspar.com}. Financially supported by the French Fondation Sciences Math\'ematiques de Paris.}}
\date{11 April 2013}
\maketitle

\begin{abstract}
  The Kuroda negative translation translates classical logic only into intuitionistic logic, not into minimal logic. We present eight variants of the Kuroda negative translation that translate classical logic even into minimal logic. The proofs of their soundness theorems are interesting because they illustrate four different methods of proof.
\end{abstract}

\begin{definition} Let $P$ range through the atomic formulas.
  \begin{enumerate}
    \item The \emph{Kuroda negative translation}~$\kc$~\cite[page~46]{Kuroda1951} translates each formula $A$ to the formula $\ku A \defequiv \neg\neg\kd A$ where $\kd A$ is defined by recursion on the length of $A$ by
    \begin{align*}
      \kd P &\defequiv P, &
      \kd{(A \to B)} &\defequiv \kd A \to \kd B, \\
      \kd{(A \wedge B)} &\defequiv \kd A \wedge \kd B, &
      \kd{(\Forall x A)} &\defequiv \Forall x \neg\neg\kd A, \\
      \kd{(A \vee B)} &\defequiv \kd A \vee \kd B, &
      \kd{(\Exists x A)} &\defequiv \Exists x \kd A.
    \end{align*}

    \item The variants $\kvc1$, $\kvc2$~\cite[page~21]{Avigad2010}, $\kvc3$, $\kvc4$, $\kvc5$~\cite[page~229]{FerreiraOliva2012}, $\kvc6$~\cite[section~6.3]{FerreiraOliva2011}, $\kvc7$ and $\kvc8$ of $\kc$ are defined analogously to $\kc$ except for
    \begin{align*}
      \kvd1 P &\defequiv P \vee \bot, &
      \kvd5{(A \to B)} &\defequiv \neg \kvd5 A \vee \kvd5 B, \\
      \kvd2 P &\defequiv \neg\neg P, &
      \kvd6{(A \to B)} &\defequiv \kvd6 A \to \neg\neg\kvd6 B, \\
      \kvd3 P &\defequiv (\bot \to P) \to P, &
      \kvd7{(A \to B)} &\defequiv \neg\kvd7 B \to \neg\kvd7 A, \\
      \kvd4{(A \to B)} &\defequiv \kvd4 A \to \kvd4 B \vee \bot, &
      \kvd8{(A \to B)} &\defequiv \neg(\kvd8 A \wedge \neg\kvd8 B).
    \end{align*}
  \end{enumerate}
\end{definition}

\begin{theorem}[soundness and characterisation] For $\kvc1$, $\kvc2$, $\kvc3$, $\kvc4$, $\kvc5$~\cite[page~229]{FerreiraOliva2012}, $\kvc6$~\cite[section~6.3]{FerreiraOliva2011}, $\kvc7$ and $\kvc8$ we have:
  \begin{enumerate}
    \item $\cl + \Gamma \vdash A \ \Rightarrow \ \ml + \kvu i \Gamma \vdash \kvu i A$;
    \item $\cl \vdash A \leftrightarrow \kvu i A$.
  \end{enumerate}
\end{theorem}

\begin{proof}
  The characterisation theorems are proved by induction on the length of $A$. Let us prove the soundness theorems.
  \begin{description}
    \item[$\kvc1,\kvc2,\kvc3,\kvc4.\ \,$] Let $\tvc1$~\cite[page~686]{Leivant1985}, $\tvc2$~\cite[page~686]{Leivant1985}, $\tvc3$ and $\tvc4$ be the translations of formulas defined by $\tvu1 P \defequiv P \vee \bot$, $\tvu2 P \defequiv \neg\neg P$, $\tvu3 P \defequiv (\bot \to P) \to P$, $\tvu4 P \defequiv P \vee \bot$, $\tvu4{(A \to B)} \defequiv \tvu4 A \to \tvu4 B \vee \bot$, $\tvc 1$, $\tvc2$ and $\tvc3$ commute with $\wedge$, $\vee$, $\to$, $\forall$ and $\exists$, and $\tvc4$ commutes with $\wedge$, $\vee$, $\forall$ and $\exists$. We can prove $\il + \Gamma \vdash A \ \Rightarrow \ \ml + \tvu i \Gamma \vdash \tvu i A$~\cite[page~686]{Leivant1985} by induction on the length of the proof of $A$, and $\ml \vdash \tvu i {(\ku A)} \leftrightarrow \kvu i A$ by induction on the length of $A$. Then $\cl + \Gamma \vdash A \ \Rightarrow \ \il + \ku\Gamma \vdash \ku A \ \Rightarrow \ \ml + \tvu i {(\ku\Gamma)} \vdash \tvu i {(\ku A)} \ \Rightarrow \ \ml + \kvu i \Gamma \vdash \kvu i A$.

    \item[$\kvc5.\ \,$] Let $\clp$ be $\cl$ based on $\neg$, $\vee$, and $\exists$~\cite[section~2.6]{Shoenfield1967}. As remarked by Benno van den Berg, $\kc$ (extended by $\kd{(\neg A)} \defequiv \neg \kd A$) translates $\clp$ into $\ml$. Let $\tvc5$ be the translation of formulas defined by $\tvu5 \bot \defequiv \neg(\neg C \vee C)$ (where $C$ is a fixed closed formula), $\tvu5 P \defequiv P$ (for $P \not\equiv \bot$), $\tvu5{(A \wedge B)} \defequiv \neg(\neg\tvu5 A \vee \neg\tvu5 B)$, $\tvu5{(A \to B)} \defequiv \neg\tvu5 A \vee \tvu5 B$, $\tvu5{(\Forall x A)} \defequiv \neg\Exists x\neg\tvu5 A$ and $\tvc5$ commutes with $\vee$ and $\exists$. We can prove $\cl + \Gamma \vdash A \ \Rightarrow \ \clp + \tvu5\Gamma \vdash \tvu5 A$ by induction on the length of the proof of $A$, and $\ml \vdash \ku{(\tvu5 A)} \leftrightarrow \kvu5 A$ by induction on the length of $A$. Then $\cl + \Gamma \vdash A \ \Rightarrow \ \clp + \tvu5\Gamma \vdash \tvu5 A \ \Rightarrow \ \ml + \ku{(\tvu5\Gamma)} \vdash \ku{(\tvu5 A)} \ \Rightarrow \ \ml + \kvu5\Gamma \vdash \kvu5 A$.

    \item[$\kvc6.\ \,$] The proof is by induction on the length of the proof of $A$ in G\"odel's system~\cite[section~1.1.4]{Troelstra1973} plus the law of excluded middle. The greatest difficulty is the rule $\frac{A \to B}{C \vee A \to C \vee B}$: its translation by $\kvc6$ is $\frac{\neg\neg(\kvd6 A \to \neg\neg\kvd6 B)}{\neg\neg(\kvd6 C \vee \kvd6 A \to \neg\neg(\kvd6 C \vee \kvd6 B))}$; from the premise we get $\kvd6 A \to \neg\neg\kvd6 B$ (by $\ml \vdash \neg\neg(D \to \neg E) \to (D \to \neg E)$), so $\kvd6 C \vee \kvd6 A \to \kvd6 C \vee \neg\neg\kvd6 B$ (by the rule), thus $\kvd6 C \vee \kvd6 A \to \neg\neg(\kvd6 C \vee \kvd6 B)$ (by $\ml \vdash D \vee \neg\neg E \to \neg\neg(D \vee E)$), getting the conclusion (by $\ml \vdash D \to \neg\neg D$).

    \item[$\kvc7,\kvc8.\ \,$] We can prove $\ml \vdash \kvd6 A \leftrightarrow \kvd7 A$ and $\ml \vdash \kvd6 A \leftrightarrow \kvd8 A$ by induction on the length of $A$, so the soundness theorems of $\kvc7$ and $\kvc8$ follow from the soundness theorem of $\kvc6$.\qedhere
  \end{description}
\end{proof}

\bibliography{References}{}
\bibliographystyle{plain}

\end{document}